\documentclass{amsart}
\usepackage[nobysame,initials]{amsrefs}
\usepackage{bbm}

\usepackage{hyperref}
\usepackage{xcolor}
\usepackage[normalem]{ulem}

\newcommand{\R}{\mathbb{R}}
\newcommand{\E}{\mathbb{E}}

\newcommand{\F}{\mathcal{F}}

\newtheorem{theorem}{Theorem}
\newtheorem{lemma}{Lemma}
\newtheorem{claim}{Claim}

\newcommand{\conv}{{\rm conv }\,}

\newcommand{\dx}{{\rm d}}
\newcommand{\pr}{\mathbb P}
\newcommand{\indi}{\mathbbm{1}}

\renewcommand{\d}{\mathrm{d}}
\newcommand{\dual}[1]{#1^{*,r}}

\DeclareMathOperator{\Var}{Var}
\DeclareMathOperator{\inti}{int}

\DeclareMathOperator{\diam}{diam}

\allowdisplaybreaks

\date{\today}

\author{Ferenc Fodor}
\address{Department of Geometry, Bolyai Institute, 
University of Szeged, Aradi v\'ertan\'uk tere 1, 6720 Szeged, Hungary}
\email{fodorf@math.u-szeged.hu}

\author{Viktor V\'{\i}gh}
\address{Department of Geometry, Bolyai Institute, 
University of Szeged, Aradi v\'ertan\'uk tere 1, 6720 Szeged, Hungary}
\email{vigvik@gmail.com}

\title{Variance estimates for random disc-polygons in smooth
convex discs}

\begin{document}

\begin{abstract}
In this paper we prove asymptotic upper bounds on the variance
of the number of vertices and missed area of inscribed random disc-polygons 
in smooth convex discs whose boundary is $C^2_+$. We also consider a circumscribed variant of this probability model in which the convex disc is approximated by the intersection of random circles.
\end{abstract}


\maketitle

\section{Introduction and results}
Let $K$ be a convex disc (compact convex set with non-empty interior) 
in the Euclidean plane $\R^2$. Assume that the boundary $\partial K$ is of class $C^2_+$, that is,
two times continuously differentiable and the curvature at every point of 
$\partial K$ is strictly positive. Let $\kappa(x)$ denote the 
curvature at $x\in\partial K$, and let $\kappa_m$ ($\kappa_M$)
be the minimum (maximum) of $\kappa (x)$ over $\partial K$. It is known, cf. \cite[Section~3.2]{Sch14}, that
in this case a closed circular disc of radius $r_m=1/\kappa_M$ rolls freely in $K$, that is, for each $x\in\partial K$, there exists a $p\in\R^2$ with $x\in r_mB^2+p\subset K$. Moreover, $K$ slides
freely in a circle of radius $r_M=1/\kappa_m$, which means that
for each $x\in\partial K$ there is a vector $p$ such that
$x\in r_M \partial B^2+p$ and $K\subset r_M B^2+p$. The latter yields that for any two points $x,y\in K$, the intersection of all closed circular discs of radius $r\geq r_M$, denoted by $[x,y]_r$ and called the $r$-spindle of $x$ and $y$, is also contained in $K$. Furthermore, for any
$X\subset K$, the intersection of all radius $r$ circles, called the closed $r$-hyperconvex hull (or $r$-hull for short) and denoted by $\conv_r(X)$, is contained
in $K$.  The concept of hyperconvexity, also called spindle convexity or $r$-convexity, has been much investigated recently. This notion naturally arises in many questions where a convex set can be represented as the intersection of equal radius closed balls. For more information and references about hyperconvexity, we refer to the paper of Bezdek, L\'angi, Nasz\'odi, Papez \cite{BLNP07}, see also a short overview in G. Fejes T\'oth, Fodor \cite{FTGF2015} and Fodor, Kurusa, V\'\i gh \cite{FoKuVi16}.  

Let $K$ be a convex disc with $C^2_+$ boundary, and let 
$X_n=\{x_1,\ldots, x_n\}$ be a sample of $n$ independent
random points chosen from $K$ according to the uniform probability 
distribution. 
The (linear) convex hull $\conv (X_n)$ is a random 
convex polygon in $K$. The geometric properties of
$\conv (X_n)$ have been investigated extensively in the 
literature. For more information on this topic and further references we refer to the surveys by B\'ar\'any \cite{Ba08}, Schneider \cites{Sch17, sch08}, Weil and Wieacker \cite{WeWi93}.

Here we examine the following random model. Let $r\geq r_M$, and let 
$K_n^r=\conv_r(X_n)$ be the $r$-hull of $X_n$, which is 
a (uniform) random disc-polygon in $K$. 
Let $f_0(K_n^r)$ denote the number 
of vertices (and also the number of edges) of $K_n^r$, and let 
$A(K_n^r)$ denote the area of $K_n^r$. 
The asymptotic behaviour of the expectation of the random variables
$A(K_n^r)$ and $f_0(K_n^r)$ were investigated by Fodor, Kevei and V\'\i gh in
\cite{FoKeVi2014}, where (among others) the following two theorems were proved.

\begin{theorem}[\cite{FoKeVi2014}]\label{thm:expectation}
Let $K$ be a convex disc whose boundary is 
of class $C^2_+$. For any $r>r_M$ it holds that                         
\begin{equation}\label{f:vertices-R}
\lim_{n\to\infty}\E(f_0(K_n^r))\cdot n^{-1/3}=\sqrt[3]{\frac 2{3A(K)}}\cdot 
\Gamma\left ( \frac 53 \right ) \int_{\partial K} 
\left (\kappa(x)-\frac{1}{r} \right)^{1/3} \dx x,
\end{equation}
and 
\begin{equation}\label{f:area-R}
\lim_{n\to\infty}\E(A(K\setminus K_n^r))\cdot n^{2/3}=
\sqrt[3]{\frac {2A(K)^2}{3}}\Gamma\left ( \frac 53 \right ) 
\int_{\partial K} \left (\kappa(x)-\frac{1}{r}
 \right)^{1/3} \dx x.
 \end{equation}
\end{theorem}

\begin{theorem}[\cite{FoKeVi2014}]\label{thm:circle}
For $r>0$ let $K=rB^2$ be the closed circular disc of
radius $r$. Then 
\begin{equation}\label{circle-sides-exp}
\lim_{n\to\infty}\E (f_0((K)_n^r)=\frac{\pi^2}{2},
\end{equation}
and
\begin{equation}
\lim_{n\to\infty}\E (K\setminus (K)_n^r)\cdot n=\frac{r^2\cdot\pi^3}{3}.
\end{equation}
\end{theorem}

$\Gamma(\cdot)$ denotes Euler's gamma function, and integration 
on $\partial K$ is with respect to arc length

Observe that in Theorem~\ref{thm:circle} the expectation $\E (f_0((K)_n^r)$ of the number of vertices tends to a constant as $n\to\infty$. This is a surprising fact that has no clear analogue in the linear case. A similar phenomenon was recently established by B\'ar\'any, Hug, Reitzner, Schneider \cite{BHRS2017} 
about the expectation of the number of facets of certain spherical random polytopes in halfspheres, see \cite[Theorem~3.1]{BHRS2017}.

We note that Theorem~1 can also be considered as a generalization of the classical asymptotic results of R\'enyi and Sulanke about the expectation of the vertex number, missed area and perimeter difference of (linear) random convex polygons in smooth convex discs, cf. \cites{ReSu63, ReSu64}, in the sense that it reproduces the formulas of R\'enyi and Sulanke in the limit as $r\to\infty$, see \cite[Section~3]{FoKeVi2014}.

Obtaining information on the second order properties of random variables associated with random polytopes is much harder than on first order properties. It is only recently that variance estimates, laws of large numbers, and central limit theorems have been proved in various models, see, for example, B\'ar\'any, Fodor, V\'\i gh \cite{BFV10}, B\'ar\'any, Reitzner \cite{BR10}, B\'ar\'any, Vu \cite{BaraVu}, Fodor, Hug, Ziebarth \cite{FHZ16},  B\"or\"oczky, Fodor, Reitzner, V\'\i gh \cite{BFRV09}, Reitzner \cites{R2003, Rei05}, Schreiber, Yukich \cite{SY}, Vu \cites{Vu05, Vu06}, and the very recent papers by Th\"ale, Turchi, Wespi \cite{TTW17}, Turchi, Wespi \cite{TW18+}. For an overview, we refer to B\'ar\'any \cite{Ba08} and Schneider \cite{Sch17}.

In this paper, we prove the following asymptotic estimates for
the variance of $f_0(K_n^r)$ and $A(K_n^r)$ in the spirit of Reitzner \cite{R2003}. 

For the order of magnitude, we use the Landau symbols: if for two functions $f,g:I\to \R$, $I\subset\R$, there is a constant $\gamma>0$ such that $|f|\leq \gamma g$ on $I$, then we write $f\ll g$ or $f=O(g)$. If $f\ll g$ and $g\ll f$, then this fact is indicated by the notation $f\approx g$.

\begin{theorem}\label{thm:variance}
With the same hypotheses as in Theorem~\ref{thm:expectation}, it holds that
\begin{equation}
\label{vertex-variance}
\Var (f_0(K_n^r))\ll n^{\frac 13},
\end{equation}
and
\begin{equation}
\label{area-variance}
\Var (A(K_n^r))\ll n^{-\frac 53},
\end{equation}
where the implied constants depend only on $K$ and $r$.  
\end{theorem}

In the special case when $K$ is the closed circular disc of radius $r$, we prove
the following.

\begin{theorem}\label{them:circle-variance}
With the same hypotheses as in Theorem~\ref{thm:circle}, it holds that
\begin{equation}\label{circle-sides}
\Var (f_0(K_n^r))\approx const.,
\end{equation}
and
\begin{equation}\label{circle-area}
\Var (A(K_n^r)))\ll n^{-2},
\end{equation}
where the implied constants depend only on $r$. 
\end{theorem}

From Theorem~\ref{thm:variance} we can conclude the following strong laws of large numbers. Since the proof follows a standard argument, see, for example, B\"or\"oczky, Fodor, Reitzner, V\'\i gh \cite[p. 2294]{BFRV09} or Reitzner \cite[Section 5]{R2003}, we omit the details.

\begin{theorem}\label{law-of-large-numbers}
With the same hypotheses as in Theorem~\ref{thm:expectation}, it holds that
\begin{equation}\label{f:vertices-R}
\lim_{n\to\infty} f_0(K_n^r)\cdot n^{-1/3}=\sqrt[3]{\frac 2{3A(K)}}\cdot 
\Gamma\left ( \frac 53 \right ) \int_{\partial K} 
\left (\kappa(x)-\frac{1}{r} \right)^{1/3} \dx x,
\end{equation}
and 
\begin{equation}\label{f:area-R}
\lim_{n\to\infty} A(K\setminus K_n^r)\cdot n^{2/3}=
\sqrt[3]{\frac {2A(K)^2}{3}}\Gamma\left ( \frac 53 \right ) 
\int_{\partial K} \left (\kappa(x)-\frac{1}{r}
 \right)^{1/3} \dx x
 \end{equation}
with probability $1$.
\end{theorem}

In the theory of random polytopes there is more information on models in which the polytopes are generated as the convex hull of random points from a convex body $K$ than on polyhedral sets produced by random closed half-spaces containing $K$. For some recent results and references in this direction see, for example, B\"or\"oczky, Fodor, Hug \cite{BFH10}, B\"or\"oczky, Schneider \cite{BS08}, Fodor, Hug, Ziebarth \cite{FHZ16} and the survey by Schneider \cite{Sch17}.

In Section~\ref{circumscribed}, we consider a model of random disc-polygons that contain a given convex disc with $C^2_+$ boundary. In this circumscribed probability model, we give asymptotic formulas for the expectation of the number of vertices of the random disc-polygon, the area difference and the perimeter difference of the random disc-polygon and $K$, cf. Theorem~6. Furthermore,  Theorem~7 provides an asymptotic upper bound on the variance of the number of vertices of the random disc-polygons. 

The outline of the paper is the following. In Section~\ref{preparations} we collect some geometric facts that are needed for the arguments. Theorem~\ref{thm:variance} is proved in Section~\ref{proofofthm3}, and Theorem~\ref{them:circle-variance} is verified in Section~\ref{proofofthm4}. In Section~\ref{circumscribed}, we discuss a different probability model in which $K$ is approximated by the intersection of random closed circular discs. This model is a kind of dual to the inscribed one.

\section{Preparations}\label{preparations}
We note that it is enough to prove Theorem~\ref{thm:variance} for the case when $r_M<1$ and $r=1$, and
Theorem~\ref{them:circle-variance} for $r=1$. 
The general statements then follow by a simple scaling argument.
Therefore, from now on we assume
that $r=1$ and to simplify notation we write $K_n$ for $K_n^1$.

Let $\overline{B}^2$ denote the open unit ball of radius $1$ centred at the origin $o$. A {\em disc-cap (of radius $1$)} of $K$ is a set of the form $K\setminus (\overline{B}^2+p)$ for some $p\in \R^2$.
   
We start with recalling the following notations from \cite{FoKeVi2014}. 
Let $x$ and $y$ be two points from $K$. 
The two unit circles passing through $x$ and
$y$ determine two dics-caps of $K$, which we denote by
$D_-(x,y)$ and $D_+(x,y)$, respectively, such that $A(D_-(x,y))\leq A(D_+(x,y))$.
For brevity of notation, we write $A_-(x,y)=A(D_-(x,y))$
and $A_+(x,y)=A(D_+(x,y))$. 
It was shown in \cite{FoKeVi2014} (cf. Lemma~3) that if the boundary of $K$ is
of class $C^2_+$ ($r_M<1$), then there exists a $\delta>0$ 
(depending only on $K$) with the property that for any $x, y\in\inti K$ it
holds that $A_+(x,y)>\delta$.

We need some further technical lemmas about general disc-caps. Let $u_x\in S^1$ denote the (unique) outer unit normal to $K$ at the boundary point $x$, and $x_u\in \partial K$ the unique boundary point with outer unit normal $u\in S^1$. 

\begin{lemma}[\cite{FoKeVi2014}, pp. 905, Lemma 4.1.]\label{vertexclaim}
	Let $K$ be a convex disc with $C^2_{\textcolor{red}{+}}$ smooth boundary and assume that $\kappa_m>1$.
	Let $D=K\setminus (\overline{B}^2+p)$ be a non-empty disc-cap of $K$ (as above).  Then there exists a unique
	point $x_0\in \partial K \cap \partial D$ such that there exists a $t\geq 0$ with $B^2+p=B^2+x_0-(1+t) u_{x_0}.$
	We refer to $x_0$ as the {\em vertex} of $D$ and to $t$ as the {\em height of $D$}. 
\end{lemma}

Let $D(u,t)$ denote the disc-cap with vertex $x_u\in\partial K$ and
height $t$. Note that for each $u\in S^1$, there exists 
a maximal positive constant $t^*(u)$ such that $(B+x_u-(1+t)u)\cap K\neq\emptyset$
for all $t\in [0, t^*(u)]$. 
For simplicity we let $A(u,t)=A(D(u,t))$ and let $\ell(u,t)$ denote 
the arc-length of $\partial D(u,t)\cap (\partial B+x_u-(1+t)u)$. 

We need some estimates for $A(u,t)$ and $\ell(u,t)$, that we recall from \cite[p. 906, Lemma 4.2]{FoKeVi2014}:
\begin{equation}\label{Vandellestimate}
\lim_{t\to 0^+}\ell(u_x,t)\cdot t^{-1/2}= 2\sqrt{\frac{2}{\kappa(x)-1}}, \quad \quad
\lim_{t\to 0^+} A(u_x,t)\cdot t^{-3/2}=\frac43 \sqrt{\frac{2}{\kappa(x)-1}}.
\end{equation}

It is clear that \eqref{Vandellestimate} imply that $A(u,t)$ and $\ell(u,t)$ satisfy the following relations uniformly in $u$:
\begin{equation}\label{Vandellestimate2}
\ell(u_x,t) \approx t^{1/2}, \quad\quad
A(u_x,t)\approx t^{3/2},
\end{equation}
where the implied constants depend only on $K$.

Let $D$ be a disc-cap of $K$ with vertex $x$. For a line $e\subset\R^2$ with $e\perp u_x$, let $e_+$ denote closed half plane containing $x$. Then there exist a maximal cap $C_-(D)=K\cap e_+\subset D$, and a minimal cap $C_+(D)=e'_+\cap K\supset D$.

\begin{claim}\label{hasznossag}
There exists a constant $\hat c$ depending only $K$ such that if the height of the disc-cap $D$ is sufficiently small, then 
\begin{equation}\label{sapkatart}
C_-(D)-x\subset \hat c (C_+(D)-x).
\end{equation}	
\end{claim}

\begin{proof}
Let us denote by $h_-$ ($h_+$) the height of $C_-(D)$ ($C_+(D)$ resp.), which is the distance of $x$ and $e$ ($e'$ resp.). By convexity, it is enough to find a constant $\hat{c}>0$ such that for all disc-caps of $K$ with sufficiently small height $h_+/h_-<\hat c$ holds.

Choose an arbitrary $R\in (1/\kappa_m,1)$, and consider $\hat B=RB^2+x-Ru_x$, the disc of radius $R$ that supports $K$ in $x$. Clearly, $\hat B \supseteq K$ implies $D=K\cap (\overline{B}^2+p)\subset (\hat B \cap (\overline{B}^2+p) =\hat D$. Also, for the respective heights $\hat h_-$ and $\hat h_+$ of $C_-(\hat D)$ and $C_+(\hat D)$, we have $\hat h_-=h_-$ and $\hat h_+>h_+$. Thus, it is enough to find $\hat c$ such that $\hat h_+/ \hat h_-<\hat c$. 
The existence of such $\hat c$ is clear from elementary geometry.
\end{proof}

Let $x_i,x_j$ ($i\neq j$) be two points from $X_n$, and let 
$B(x_i,x_j)$ be one of the unit discs that contain
$x_i$ and $x_j$ on its boundary. The shorter arc of
$\partial B(x_i,x_j)$ forms an edge of $K_n$ if the entire set $X_n$
is contained in $B(x_i,x_j)$. Note that it may happen that 
the pair $x_i,x_j$ determines two edges of $K_n$ if the above
condition holds for both unit discs that contain $x_i$ and $x_j$
on its boundary.  

First, we note that for the proof of Theorem~\ref{thm:variance}, similar to Reitzner \cite{R2003}, we may
assume that the Hausdorff distance $d_H(K, K_n)$ of $K$ and $K_n$ is at most $\varepsilon_K$, where $\varepsilon_K>0$ is a suitably chosen constant.
This can be seen the following way. Assume that $d_H(K, K_n)\geq \varepsilon_K$. 
Then there exists a point $x$
on the boundary of $K_n$ such that $\varepsilon_KB^2+x\subset K$.
There exists a supporting circle of $K_n$ through $x$ that determines
a disc-cap of height at least $\varepsilon_K$. By the above remark, 
the probability content of this disc-cap is at least $c_K>0$, where $c_K$
is a suitable constant depending on $K$ and $\varepsilon_K$. Then
$$
\pr(d_H(K, K_n)\geq \varepsilon_K)\leq \left (1-c_K\right )^{n}. 
$$

Our main tool in the variance estimates is the Efron-Stein inequality
\cite{ES1981}, 
which has previously 
been used to provide upper estimates on the variance of various 
geometric quantities associated with random polytopes in convex bodies,
cf. Reitzner \cite{R2003}. 
For more on this topic and further references we refer to the recent survey articles by B\'ar\'any \cite{Ba08} and Schneider \cite{Sch17}.

\section{Proof of Theorem~\ref{thm:variance}}\label{proofofthm3}
We present the proof of the upper bound on the variance of the vertex number in detail, and only indicate the modifications needed to prove the upper bound on the missed area. Our argument is similar to the one in Reitzner \cite[Sections~4 and 6]{R2003}.

For the number of vertices of $K_n$, the Efron-Stein inequality \cite{ES1981} states the following
\begin{equation}\label{ESi}
\Var f_0(K_n)\leq (n+1) \E (f_0(K_{n+1})-f_0(K_n))^2. 
\end{equation}

Let $x$ be an arbitrary point of $K$
and  let $x_ix_j$ be an edge of $K_n$. 
Following Reitzner \cite{R2003}, we say that the edge $x_ix_j$ 
is visible from $x$ if $x$ is not contained in $K_n$ and 
it is not contained in the unit disc of the edge $x_ix_j$.
For a point $x\in K\setminus K_n$, let $\F_n(x)$ denote the set of 
edges of $K_n$ that can be seen from $x$, and for 
$x\in K_n$ set $\F_n(x)=\emptyset$. Let $F_n(x)=|\F_n(x)|$.  

Let $x_{n+1}$ be a uniform random point in $K$ chosen independently from 
$X_n$. 
If $x_{n+1}\in K_n$, then $f_0(K_{n+1})=f_0(K_n)$. 
If, on the other hand, 
$x_{n+1}\not\in K_n$, then 
\begin{align*}
f_0(K_{n+1})& =f_0(K_n)+1-(F_n(x_{n+1})-1)\\
& =f_0(K_n)-F_n(x_{n+1})+2.
\end{align*}
Therefore,
$$|f_0(K_{n+1})-f_0(K_n)|\leq 2 F_n(x_{n+1}),$$
and by the Efron--Stein jacknife inequality
\begin{align}\label{ESspec}
\Var (f_0(K_n)) & \leq (n+1)\E (f_0(K_{n+1})-f_0(K_n))^2\\
& \leq 4(n+1)\E(F_n^2(x_{n+1})).\notag
\end{align}

Similar to Reitzner, we introduce the following notation  
(cf. \cite{R2003} p. 2147). Let $I=(i_1,i_2), i_1\neq i_2$, $i_1, i_2\in \{1,2,\ldots \}$ be an ordered pair of indices. 
Denote by $F_I$ the shorter arc of the
unique unit circle incident with $x_{i_1}$ and $x_{i_2}$ on
which $x_{i_1}$ follows $x_{i_2}$ in the positive cyclic
ordering of the circle.

Let $\indi(A)$ denote the indicator function of the event $A$.



We wish to estimate the expectation $\E (F_n^2(x_{n+1}))$ under the condition
that $d_H(K, K_n)<\varepsilon_K$. To compensate for the cases in which 
$d_H(K, K_n)\geq \varepsilon_k$, we add an error term $O((1-c_K)^n)$.  
\begin{align}
\E (& F_n(x_{n+1})^2) = \frac 1{A(K)^{n+1}}\int_K\cdots \int_K \left (\sum_I \indi (F_I\in\F_{n}(x_{n+1}))\right )^2 \d X_n\d x_{n+1}\nonumber\\
& = \frac 1{A(K)^{n+1}} \int_K\cdots \int_K \left (\sum_I \indi (F_I\in\F_{n}(x_{n+1}))\right ) \nonumber \\
& \quad \quad \times \left (\sum_J \indi (F_J\in\F_{n}(x_{n+1}))\right ) \d X_n \d x_{n+1}\nonumber\\
&\leq \frac 1{A(K)^{n+1}}\sum_I\sum_J \int_K\cdots \int_K \indi (F_I\in\F_{n}(x_{n+1}))
\indi (F_J\in\F_{n}(x_{n+1}))\nonumber\\
&\quad\quad\times\indi(d_H(K,K_n)\leq\varepsilon_K) \d X_n \d x_{n+1}
+O((1-c_K)^n)\nonumber
\end{align}
Choose $\varepsilon_K$ so small that $A(K\setminus K_n)<\delta$. 
Note that with this choice of $\varepsilon_K$ only one
of the two shorter arcs determined by $x_{i_1}$ and $x_{i_2}$ can
determine an edge of $K_n$. 

Now we fix the number $k$ of common elements of $I$ and $J$, that is, $|I\cap J|=k$. 
Let $F_1$ denote one of the shorter arcs spanned by $x_1$ and $x_2$, and
let $F_2$ be one of the shorter arcs determined by $x_{3-k}$ and $x_{4-k}$. 
Since the random points are independent, we have that
\begin{align}
&\ll \frac 1{A(K)^{n+1}} \sum_{k=0}^{2}{n\choose 2}{2\choose k}{n-2\choose 2-k} \int_K\cdots\int_K \indi(F_1\in\F_n(x_{n+1}))+\notag\\
&\quad\quad\times\indi(F_2\in\F_n(x_{n+1}))\indi(d_H(K,K_n)\leq\varepsilon_K)\d X_n\d x_{n+1}+O((1-c_K)^n)\notag\\
&\ll \frac 1{A(K)^{n+1}} \sum_{k=0}^{2}n^{4-k} \int_K\cdots\int_K \indi(F_1\in\F_n(x_{n+1}))\notag\\
&\quad\quad\times\indi(F_2\in\F_n(x_{n+1}))\indi(d_H(K,K_n)\leq\varepsilon_K)\d X_n\d x_{n+1}+O((1-c_K)^n). \label{masodik}
\end{align}

Since the roles of $F_1$ and $F_2$ are symmetric, we may assume that 
$\diam C_+(D_1)\ge \diam C_+(D_2)$, where $D_1=D_-(x_1,x_2)$ and $D_2= D_-(x_{3-k}, x_{4-k})$ are the corresponding disc-caps. Thus,
\begin{multline*}
\ll \frac 1{A(K)^{n+1}} \sum_{k=0}^{2}n^{4-k}\int_K\cdots\int_K \indi(F_1\in\F_n(x_{n+1}))\\
\times\indi(F_2\in\F_n(x_{n+1}))
\indi (\diam C_+(D_1)\ge \diam C_+(D_2))\\
\times\indi(d_H(K,K_n)\leq\varepsilon_K)\d X_n\d x_{n+1}+O((1-c_K)^n). 
\end{multline*}
Clearly, $x_{n+1}$ is a common point of the disc caps $D_1$ and $D_2$, so we may write that 
\begin{align*}
\leq &\frac 1{A(K)^{n+1}} \sum_{k=0}^{2}n^{4-k}\int_K\cdots\int_K \indi(F_1\in\F_n(x_{n+1}))\\
&\times\indi(D_1\cap D_2\neq\emptyset)
\indi (\diam C_+(D_1)\ge \diam C_+(D_2))\\
&\times\indi(d_H(K,K_n)\leq\varepsilon_K)\d X_n\d x_{n+1}+O((1-c_K)^n). 
\end{align*}
In order for $F_1$ to be an edge of $K_n$, it is necessary that $x_{5-k},\ldots x_n\in K\setminus D_1$, and for $F_1\in\F_n(x_{n+1})$ $x_{n+1}$ must be in $D_1$. Therefore
\begin{align}
\nonumber \ll &\frac 1{A(K)^{n+1}} \sum_{k=0}^{2}n^{4-k} \int_K\cdots\int_K (A(K)-A(D_1)))^{n-4+k}A(D_1)\\
&\nonumber\times\indi(D_1\cap D_2\neq\emptyset) 
\indi (\diam C_+(D_1)\ge \diam C_+(D_2))\\
&\notag\times\indi(d_H(K,K_n)\leq\varepsilon_K)\d x_1\cdots \d x_{4-k}+O((1-c_K)^n)\\
\label{kellszamformula} \ll&  \sum_{k=0}^{2}n^{4-k} \int_K\cdots\int_K \left (1-\frac{A(D_1)}{A(K)} \right )^{n-4+k}\frac{A(D_1)}{A(K)}\\
&\notag\times\indi(D_1\cap D_2\neq\emptyset)
\indi (\diam C_+(D_1)\ge \diam C_+(D_2))\\
&\times\notag\indi(d_H(K,K_n)\leq\varepsilon_K)\d x_1\cdots \d x_{4-k}+O((1-c_K)^n). \nonumber 
\end{align} 

Reitzner proved (see \cite[pp. 2149--2150]{R2003}) that if $D_1\cap D_2\neq\emptyset$, $d_H(K,K_n)\leq\varepsilon_K$ and $\diam C_+(D_1)\ge \diam C_+(D_2)$ then there exists a constant $\bar c$ (depending only on $K$) such that $C_+(D_2)\subset \bar c(C_+(D_1)-x_{D_1})+x_{D_1}$, where $x_{D_1}$ is the vertex of $D_1$. Combining this with Claim \ref{hasznossag} we obtain that there is a constant $c_1$ depending only on $K$, such that $D_2\subset c_1(D_1-x_{D_1})+x_{D_1}$. Hence 
$A(D_2)\le c_1^2 A(D_1)$, and therefore
\begin{align*}
\int_K\cdots\int_K \indi(D_1\cap D_2\neq\emptyset)
\indi (\diam C_c(D_1)\ge \diam C_c(D_2))\\
\times \indi(d_H(K,K_n)\leq\varepsilon_K)\d x_3\cdots \d x_{4-k}\ll A(D_1)^{2-k}. 
\end{align*}
We continue by estimating (\ref{kellszamformula}) term by term (omitting the $O((1-c_K)^n)$ term).
\begin{align*}
& n^{4-k} \int_K\cdots\int_K \left (1-\frac{A(D_1)}{A(K)} \right )^{n-4+k}\frac{A(D_1)}{A(K)}  \indi(D_1\cap D_2\neq\emptyset)\\
& \quad\quad\times
\indi (\diam C_c(D_1)\ge \diam C_c(D_2))\indi(d_H(K,K_n)\leq\varepsilon_K)\d x_1\cdots \d x_{4-k} \\
\ll& n^{4-k} \int_K \int_K \left (1-\frac{A(D_1)}{A(K)} \right )^{n-4+k}\left (\frac{A(D_1)}{A(K)} \right) ^{3-k} \indi(d_H(K,K_n)\leq\varepsilon_K) \d x_1\d x_2. 
\end{align*}

Now, we use the following parametrization of $(x_1,x_2)$ the same way as in \cite{FoKeVi2014} to transform the integral.
Let 
\begin{equation}\label{kulcstrafo}
(x_1, x_2)=\Phi (u, t, u_1, u_2),
\end{equation}
where $u, u_1, u_2\in S^1$ and $0\leq t\leq t_0(u)$ are chosen such that 
$$D(u,t)=D_1=D_-(x_1, x_2),$$ 
and
$$(x_1, x_2)=(x_u-(1+t)u+u_1, x_u-(1+t)u+u_2).$$

More information on this transformation can be found in \cite[pp. 907-909.]{FoKeVi2014}, here we just recall that the Jacobian of $\Phi$ is
 \begin{equation}\label{jacobi}
 |J\Phi|=\left ( 1+t-\frac{1}{\kappa (x_u)}\right )|u_1\times u_2|,
 \end{equation} 
where $u_1\times u_2$ denotes the cross product of $u_1$ and $u_2$.
 
 Let $L(u,t)=\partial D_1\cap \mathrm{int } K$, then we obtain that
\begin{align}
&\nonumber \ll n^{4-k} \int_{S^1} \int_0^{t^*(u)}\int_{L(u,t)} \int_{L(u,t)} \left  (1-\frac{A(u,t)}{A(K)} \right )^{n-4+k}\left (\frac{A(u,t)}{A(K)}\right )^{3-k} \\ 
&\nonumber \times\left ( 1+t-\frac{1}{\kappa(x_u)}\right ) |u_1\times u_2|\dx u_1 \dx u_2 \dx t \dx u\\
&\nonumber = n^{4-k} \int_{S^1} \int_0^{t^*(u)} \left  (1-\frac{A(u,t)}{A(K)} \right )^{n-4+k}\left (\frac{A(u,t)}{A(K)}\right )^{3-k} \\ 
& \times\left ( 1+t-\frac{1}{\kappa(x_u)}\right )
 (\ell(u,t)-\sin \ell(u,t)) \dx t \dx u.
\end{align}

From now on the evaluation follows a standard way, thus we only sketch the major steps.

 First, we split the domain of integration with respect to $t$ into two parts. Let $h(n)=(c \ln n/n)^{2/3}$, where $c$ is a sufficiently large absolute constant. Using (\ref{Vandellestimate2}), one can easily see that
\begin{align*}
&n^{4-k}\int_{S^1} \int_{h(n)}^{t^*(u)}\left  (1-\frac{A(u,t)}{A(K)} \right )^{n-4+k}\left (\frac{A(u,t)}{A(K)}\right )^{3-k}\\ 
&\times\left ( 1+t-\frac{1}{\kappa(x_u)}\right ) (\ell(u,t)-\sin \ell(u,t)) \dx t \dx \ll n^{-2/3}.
\end{align*}

Therefore, it is enough to estimate
\begin{align*}
& n^{4-k} \int_{S^1} \int_0^{h(n)} \left  (1-\frac{A(u,t)}{A(K)} \right )^{n-4+k}\left (\frac{A(u,t)}{A(K)}\right )^{3-k} \\ 
& \times\left ( 1+t-\frac{1}{\kappa(x_u)}\right )
 (\ell(u,t)-\sin \ell(u,t)) \dx t \dx u.
\end{align*}

Using (\ref{Vandellestimate2}), the fact that $\kappa(x)>1$ for all $x\in\partial K$, and the Taylor series of the $\sin$ function, furthermore, assuming that $n$ is large enough, we obtain that
\begin{align*}
& \ll n^{4-k} \int_{S^1} \int_0^{h(n)} \left  (1- c_K t^{3/2} \right )^{n-4+k}\left (t^{3/2}\right )^{3-k} \cdot 1 \cdot t^{3/2}  \dx t \dx u\\
& \ll n^{4-k} \int_0^{h(n)} \left  (1- c_K t^{3/2} \right )^{n-4+k}t^{\frac{12-3k}{2}} \dx t\ll n^{-2/3},
\end{align*}
where in the last step we applied \cite[pp. 2290, formula (11)]{BFRV09}. Together with (\ref{ESspec}), this yields the desired upper estimate for $\Var f_0(K_n)$.

As the the argument for the case of the missing area is very similar, we only indicate the major steps.

Again, we use the Efron-Stein inequality \cite{ES1981}, which states the following for the missed area
\begin{equation}\label{ESi2}
\Var A(K\setminus K_n)\leq (n+1) \E (A(K_{n+1})-A(K_n))^2. 
\end{equation}
Therefore, we need to estimate $\E(A(K_{n+1})-A(K_n))^2$. Following the ideas of Reitzner \cite{R2003}, one can see that 
\begin{align*}
& \E(A(K_{n+1})-A(K_n))^2 \ll \sum_I \sum_J \int_K\cdots\int_K \indi(F_1\in\F_n(x_{n+1})) A(D_1) \\ 
& \quad\quad\times \indi(F_2\in\F_n(x_{n+1})) A(D_2) \indi(d_H(K,K_n)\leq\varepsilon_K)\d X_n\d x_{n+1}.
\end{align*}
From here, we may closely follow the proof of (\ref{vertex-variance}), the only major difference is the extra $A(D_1)A(D_2)\leq A^2(D_1)$ factor in the integrand. After similar calculations as for the vertex number, we obtain that
\begin{align*}
& \ll n^{4-k} \int_{S^1} \int_0^{h(n)} \left  (1-\frac{A(u,t)}{A(K)} \right )^{n-4+k}\left (\frac{A(u,t)}{A(K)}\right )^{5-k} \\ 
& \times\left ( 1+t-\frac{1}{\kappa(x_u)}\right )
 (\ell(u,t)-\sin \ell(u,t)) \dx t \dx u.\\
& \ll n^{4-k} \int_0^{h(n)} \left  (1- c_K t^{3/2} \right )^{n-4+k}t^{\frac{20-3k}{2}} \dx t\ll n^{-8/3},
\end{align*}
which proves (\ref{area-variance}) (the missing factor $n$ comes from the Efron-Stein inequality).

\section{The case of the circle}\label{proofofthm4}
In this section we prove Theorem~\ref{them:circle-variance}. In particular, we give a detailed proof of the estimate \eqref{circle-sides} for the variance of the number of vertices of the random disc-polygon, and we only point out the necessary modifications that are needed to verify \eqref{circle-area}.

Without loss of generality, we may assume that $K=B^2$, and that $r=1$.
 
We begin by recalling from \cite{FoKeVi2014} that for any $u\in S^1$ and $0\leq t\leq 2$, it holds that
\begin{equation}
\ell (u,t)=2\arcsin \sqrt{1-\frac{t^2}{2}},
\end{equation}
and
\begin{equation}
A(u,t)=A(t)=t\sqrt{1-\frac{t^2}{2}}+2\arcsin\frac{t}{2}.
\end{equation}

\begin{proof}[Proof of Theorem~\ref{them:circle-variance} \eqref{circle-sides}]
From \eqref{circle-sides-exp} and Chebyshev's inequality, it follows that
$$1=\pr\left (\left |f_0(K_n^1)-\frac{\pi^2}{2}\right |>0.05\right )\leq \frac{\Var(f_0(K_n^1))}{0.05^2},$$
thus 
$$\Var(f_0(K_n^1))\geq 0.05^2.$$
This proves that $\Var(f_0(K_n^1))\gg const.$.

In order to prove the asymptotic upper bound in \eqref{circle-sides}, we use a modified version of the argument of the previous section. With the same notation as in Section~3, the Efron-Stein inequality for the vertex number yields that 
$$\Var(f_0(K_n^1))\ll n\E (F_n(x_{n+1}))^2.$$

Following a similar line of argument as above, we obtain that
\begin{align*}
n&\E(F_n(x_{n+1}))^2=\frac n{\pi^{n+1}}\int_{(B^2)^{n+1}}\left (\sum_{I}\indi(F_I\in \F_n(x_ {n+1}))\right )\\
& \quad \quad \times
\left (\sum_{J}\indi(F_J\in \F_n(x_ {n+1}))\right ) \d x_1\cdots \d x_n\d x_{n+1}\\
&\leq \frac n{\pi^{n+1}} \sum_{I}\sum_{J}
\int_{(B^2)^{n+1}}\indi(F_I\in \F_n(x_ {n+1}))
\indi(F_J\in \F_n(x_ {n+1})) \d x_1\cdots \d x_n\d x_{n+1}
\end{align*}

Now, let $|I\cap J|=k$, where $k=0,1,2$, and let $F_1=x_1x_2$ and $F_2=x_{3-k}x_{4-k}$. 
By the independence of the random points (and by also taking into account their order), we get that
\begin{align*}
&\ll \frac n{\pi^{n+1}} \sum_{k=0}^{2}{n\choose 2}{2\choose k}{n-2\choose 2-k} \int_{(B^2)^{n+1}}\indi(F_1\in\F_n(x_{n+1}))\\
 &\quad\quad\quad\quad\quad\times \indi(F_2\in\F_n(x_{n+1}))\d x_1\cdots\d x_n\d x_{n+1}. \\
&\ll \frac 1{\pi^{n+1}}\sum_{k=0}^{2}n^{5-k} \int_{(B^2)^{n+1}} \indi(F_1\in\F_n(x_{n+1}))\indi(F_2\in\F_n(x_{n+1}))\d x_1\cdots\d x_n\d x_{n+1}.
\end{align*}
By symmetry, we may also assume that $A(D_1)\geq A(D_2)$, therefore
\begin{align*}
	&\ll \sum_{k=0}^{2} n^{5-k} \int_{(B^2)^{n+1}} \indi(F_1\in\F_n(x_{n+1}))\indi(F_2\in\F_n(x_{n+1}))\\
	&\quad\quad\quad\quad\quad \times\indi (A(D_1)\geq A(D_2)\d x_1\cdots\d x_n\d x_{n+1}.
\end{align*}
By integrating with respect to $x_{5-k}, \ldots, x_n$ and $x_{n+1}$ we obtain that
\begin{align*}
&\ll\sum_{k=0}^{2}n^{5-k}\int_{B^2}\cdots\int_{B^2}
\left (1-\frac{A(D_1)}{\pi}\right )^{n-4+k}\frac{A(D_1)}{\pi}\\
&\quad\quad\quad\quad\quad \times\indi (A(D_1)\geq A(D_2))
\d x_1\cdots \d x_{4-k}
\end{align*}
If $A(D_1)\geq A(D_2)$, then $D_2$ is fully contained in the circular annulus whose width is equal to the height of the disc-cap $D_1$. The area of this annulus not more than $2A(D_1)$. Therefore,
\begin{align*}
&\ll\sum_{k=0}^{2}n^{5-k}\int_{B^2}\int_{B^2}
\left (1-\frac{A(D_1)}{\pi}\right )^{n-4+k}A(D_1)^{3-k}\d x_1\d x_2.
\end{align*}
As common in these arguments, we may assume that $A(D_1)/\pi<c\log n/n$ for some suitable constant $c$ that will be determined later. To see this, let $A(D_1)/\pi\geq c\log n/n$. Then
\begin{align*}
&\left (1-\frac{A(D_1)}{\pi}\right )^{n-4+k}A(D_1)^{3-k}\\
&\leq \left (\frac{\pi c\log n}{n}\right )^{3-k}\cdot\exp \left (-\frac{c(n-4+k)\log n}{n}\right )\\
&\ll \left (\frac{\log n}{n}\right )^{3-k}\cdot n^{-c}\\
&\ll n^{-c}.
\end{align*}
If $c>0$ is sufficiently large, then the contribution of the case when $A(D_1)/\pi\geq c\log n/n$ is $O(n^{-1})$. Thus, 
\begin{multline*}
n\E(F_n(x_{n+1}))
\ll\sum_{k=0}^{2}n^{5-k}\int_{B^2}\int_{B^2}
\left (1-\frac{A(D_1)}{\pi}\right )^{n-4+k}A(D_1)^{3-k}\\
\times\indi(A(D_1)\leq c\log n /n)\d x_1\d x_2 +O(n^{-1}).
\end{multline*}
Now, we use the same type of reparametrization as in the previous section. Let $(x_1,x_2)=(-t u_1,-tu_2)$, $u\in S^1$ and $0\leq t<c\log n/n$. Then
\begin{align*}
	&\ll \sum_{k=0}^{2}n^{5-k}
	\int_{S^1}\int_{0}^{c^*\log n/n}\int_{S^1}\int_{S^1}
	\left (1-\frac{A(u,t)}{\pi}\right )^{n-4+k}A(u,t)^{3-k}\\
	&\times t |u_1\times u_2|\d u_1\d u_2 \d u \d t +O(n^{-1})\\
	&\ll \sum_{k=0}^{2}n^{5-k}
	\int_{0}^{c^*\log n/n}
	\left (1-\frac{A(u,t)}{\pi}\right )^{n-4+k}A(u,t)^{3-k}\\
	&\times t (l(t)-\sin l(t))\d t +O(n^{-1}).
\end{align*}
Using that $l(t)\to \pi$ as $t\to 0^+$, and the Taylor series of $V(u,t)$ at $t=0$, we obtain that there exists a constant $\omega>0$ such that
\begin{align*}
	&\ll \sum_{k=0}^{2}n^{5-k}
	\int_{0}^{c^*\log n/n}
	\left (1-\omega t\right )^{n-4+k}t^{4-k}\d t +O(n^{-1})
\end{align*}	
Now, using the well-known formula for beta integrals (cf. \cite[pp. 2290, formula (11)]{BFRV09}), we obtain that
\begin{align*}	
	&\ll \sum_{k=0}^{2}n^{5-k} n^{-(5-k)}+O(n^{-1})\\
	&\ll const,
\end{align*}
which finishes the proof of the upper bound in \eqref{circle-sides}.
\end{proof}

In order to prove the asymptotic upper bound \eqref{circle-area}, only slight modifications are needed in the above argument.

\section{A circumscribed model}\label{circumscribed}
 In the section we consider circumscribed random disc-polygons. 
  Let $K\subset\R^2$ be a convex disc with $C^2_+$ smooth boundary, and $r\geq\kappa_m^{-1}$. Consider the following set
  $$K^{*,r}=\left \{ x\in \R^2 \; | \; K\subseteq rB^2+x\right \},$$ 
  which is also called the $r$-hyperconvex dual, or $r$-dual for short, of $K$. It is known that $K^{*,r}$ is a convex disc with $C^2_+$ boundary, and it also has the property that the curvature is at least $1/r$ at every boundary point. For further information see \cite{FoKuVi16} and the references therein.
  
  For $u\in S^1$, let $x(K,u)\in \partial K$ ($x(K^{*,r},u)\in \partial K^{*,r}$ resp.) the unique point on $\partial K$ ($\partial K^{*,r}$ resp.), where the outer unit normal to $K$ ($K^{*,r}$ resp.) is $u$. For a convex disc $K\subset\R^2$ with $o\in\inti K$, let $h_K(u)=\max_{x\in K} \langle x,u\rangle$ denote the support function of $K$. Let $Per(\, \cdot \,)$ denote the perimeter.
  
  The following Lemma sums up some results from \cite[Section 2]{FoKuVi16}.
\begin{lemma}{\cite{FoKuVi16}}\label{dualprop} 
	With the notation above
\begin{enumerate}
\item $h_K(u)+h_{K^{*,r}}(-u)=r$ for any $u\in S^1$,
\item $\kappa_K^{-1}(x(u, K))+\kappa_{K^{*,r}}^{-1}(x(-u, K^{*,r}))=r$ for any $u\in S^1$,
\item $Per (K)+ Per (K^{*,r})= 2r\pi$,
\item $A(K^{*,r})=A(K)-r \cdot Per(K)+r^2\pi.$
\end{enumerate}
\end{lemma}

Now, we turn to the probability model. Let $K$ be a convex disc with $C^2_+$ boundary, and let $r>\kappa_m^{-1}$ as before.
Let $X_n=\{x_1,\ldots, x_n\}$ be a sample of $n$ independent
random points chosen from $K^{*,r}$ according to the uniform probability 
distribution, and define
\[K^{*,r}_{(n)}=\bigcap_{x\in X_n} rB^2+x. \]

$K^{*,r}_{(n)}$ is a random disc-polygon that contains $K$. Observe that, by definition $K^{*,r}_{(n)}=(\conv_r(X_n))^{*,r}$, and consequently $f_0(K^{*,r}_{(n)})=f_0(\conv_r(X_n))$. We note that this is a very natural approach to define a random disc-polygon that is circumscribed about $K$ that has no clear analogy in linear convexity. (If one takes the limit as $r\to \infty$, the underlying probability measures do not converge.) The model is of special interest in the case $K=K^{*,r}_{(n)}$, which happens exactly when $K$ is of constant width $r$.

\begin{theorem}\label{circformulas}
Assume that $K$ has $C^2_+$ boundary, and let $r>\kappa_m^{-1}$. With the notation above
\begin{align}\label{vertexexpcirc}
\lim_{n\to\infty}\E(f_0(K^{*,r}_{(n)}))\cdot n^{-1/3}= & \sqrt[3]{\frac {2r}{3(A(K)-r \cdot Per(K)+r^2\pi)}}\times \\
 \nonumber & \Gamma\left ( \frac 53 \right ) \int_{\partial K} 
\left (\kappa(x)-\frac{1}{r} \right)^{2/3} \dx x.
\end{align}

Furthermore if $K$ has $C^5_+$ boundary, then 

\begin{align*}
\lim_{n\to \infty} n^{2/3}\cdot \left (Per K^{*,r}_{(n)} - Per K \right ) = & 
\frac{(12(A(K)-r \cdot Per(K)+r^2\pi))^{2/3}}{36}\cdot  \Gamma \left (\frac 23 \right ) \\
&\times r^{-2/3} \int_{\partial K} \left ( \kappa(x) -\frac 1r \right )^{-1/3} \left (4 \kappa (x) - \frac 1r \right ) \d x;\\
\lim_{n\to \infty} n^{2/3}\cdot A(K^{*,r}_{(n)}\backslash K)= & 
\frac{(12(A(K)-r \cdot Per(K)+r^2\pi))^{2/3}}{12} \times  \\
& \Gamma \left (\frac 23 \right )\cdot r^{-2/3} \int_{\partial K} \left ( \kappa(x) -\frac 1r \right )^{-1/3} \d x.
\end{align*}
\end{theorem}

\begin{proof}
By Lemma~\ref{dualprop} it follows that $K^{*,r}$ has also $C^2_+$ boundary. As $f_0(K_{(n)}^{*,r})=f_0(\conv_r(X_n))$, we immediately get from \cite[Theorem 1.1]{FoKeVi2014} that 

\begin{equation}
\lim_{n\to\infty}\E(f_0(K_{(n)}^{*,r}))\cdot n^{-1/3}= \sqrt[3]{\frac 2{3 A(K^{*,r})}}\cdot 
\Gamma\left ( \frac 53 \right ) \int_{\partial K^{*,r}} 
\left (\kappa(x)-\frac{1}{r} \right)^{1/3} \dx x.
\end{equation}

Using Lemma~\ref{dualprop}, we proceed as follows
\begin{align*}
\int_{\partial \dual K} 
\left (\kappa(x)-\frac{1}{r} \right)^{1/3} \dx x & =\int_{S^1} \frac{\left ( \kappa(x(\dual K, u))-\frac{1}{r} \right)^{1/3}}{\kappa(x(\dual K, u))}  \d u = \\
\int_{S^1} \frac{\left ( \frac{\kappa(x(K, -u))}{r\kappa(x(K, -u))-1}-\frac{1}{r} \right)^{1/3}}{\frac{\kappa(x(K, -u))}{r\kappa(x(K, -u))-1}}  \d u & =\int_{S^1} r^{1/3} \frac{\left ( \kappa(x(K, u))-\frac{1}{r} \right)^{2/3}}{\kappa(x(K, u))}\d u \\ & =r^{1/3}\int_{\partial K} 
\left (\kappa(x)-\frac{1}{r} \right)^{2/3} \dx x.
\end{align*}

Together with Lemma~\ref{dualprop}, this proves \eqref{vertexexpcirc}.

The rest of the theorem can be proved similarly, by using \cite[Theorem 1.1 and Theorem 1.2]{FoKeVi2014}, and Lemma~\ref{dualprop}.

\end{proof}

As an obvious consequence of Theorem~\ref{thm:variance}, Lemma~\ref{dualprop}, and the definition of $K^{*,r}_{(n)}$, we obtain the following theorem.

\begin{theorem}
Assume that $K$ has $C^2_+$ boundary, and let $r>\kappa_m^{-1}$. With the notation above
$$\Var (f_0(K^{*,r}_{(n)})) \ll n^{1/3}.$$
\end{theorem}

{\it Remark.} We note that if $K$ is a convex disc of constant width $r$, then $K^{*,r}=K$ (see e.g. \cite{FoKuVi16}), and similar calculations to those in the proof of Theorem~\ref{circformulas} provide some interesting integral formulas. For example,  for a real $p$ we obtain that
$$\int_{\partial K} 
\left (\kappa(x)-\frac{1}{r} \right)^{p} \dx x = r^{1-2p}\int_{\partial K} 
\left (\kappa(x)-\frac{1}{r} \right)^{1-p} \dx x.$$

\section{Acknowledgements}
The research of the authors was partially supported by the National Research,
Development and Innovation Office of Hungary NKFIH 116451 grant. 
V. V\'\i gh was supported by the J\'anos Bolyai Research Scholarship of the Hungarian Academy of Sciences. This research was also supported by the EU-funded Hungarian grant
EFOP-3.6.2-16-2017-00015.


\begin{bibdiv}[References]
\begin{biblist}

\bib{Ba08}{article}{
   author={B{\'a}r{\'a}ny, Imre},
   title={Random points and lattice points in convex bodies},
   journal={Bull. Amer. Math. Soc. (N.S.)},
   volume={45},
   date={2008},
   number={3},
   pages={339--365},
}

\bib{BFV10}{article}{
	author={B{\'a}r{\'a}ny, I.},
	author={Fodor, F.},
	author={V{\'{\i}}gh, V.},
	title={Intrinsic volumes of inscribed random polytopes in smooth convex
		bodies},
	journal={Adv. in Appl. Probab.},
	volume={42},
	date={2010},
	number={3},
	pages={605--619},
}

\bib{BHRS2017}{article}{
	author={B\'ar\'any, Imre},
	author={Hug, Daniel},
	author={Reitzner, Matthias},
	author={Schneider, Rolf},
	title={Random points in halfspheres},
	journal={Random Structures Algorithms},
	volume={50},
	date={2017},
	number={1},
	pages={3--22},
	issn={1042-9832},
}

\bib{BL89}{article}{
	author={I. B\'ar\'any},
	author={D. Larman},
	title={Convex bodies, economic cap coverings, random polytopes},
	journal={Mathematika},
	volume={35},
	date={1988},
	pages={274--291.},
}

\bib{BR10}{article}{
	author={B{\'a}r{\'a}ny, Imre},
	author={Reitzner, Matthias},
	title={On the variance of random polytopes},
	journal={Adv. Math.},
	volume={225},
	date={2010},
	number={4},
	pages={1986--2001},
	issn={0001-8708},
}

\bib{BaraVu}{article}{
	author={B{\'a}r{\'a}ny, Imre},
	author={Vu, Van},
	title={Central limit theorems for Gaussian polytopes},
	journal={Ann. Probab.},
	volume={35},
	date={2007},
	number={4},
	pages={1593--1621},
	issn={0091-1798},
}

\bib{BLNP07}{article}{
   author={K. Bezdek},
   author={Z. L{\'a}ngi},
   author={M. Nasz{\'o}di},
   author={P. Papez},
   title={Ball-polyhedra},
   journal={Discrete Comput. Geom.},
   volume={38},
   date={2007},
   number={2},
   pages={201--230},
}


\bib{BFH10}{article}{
	author={B\"or\"oczky, K\'aroly J.},
	author={Fodor, Ferenc},
	author={Hug, Daniel},
	title={The mean width of random polytopes circumscribed around a convex
		body},
	journal={J. Lond. Math. Soc. (2)},
	volume={81},
	date={2010},
	number={2},
	pages={499--523},
	issn={0024-6107},
}


\bib{BFRV09}{article}{
author={K. J. B\"or\"oczky},
author={F. Fodor},
author={M. Reitzner},
author={V. V\'igh},
title={Mean width of random polytopes in a reasonably smooth convex body},
journal={J. Multivariate Anal.},
volume={100},
date={2009},
pages={2287--2295.},
}

\bib{BS08}{article}{
	author={B{\"o}r{\"o}czky, K{\'a}roly J.},
	author={Schneider, Rolf},
	title={The mean width of circumscribed random polytopes},
	journal={Canad. Math. Bull.},
	volume={53},
	date={2010},
	number={4},
	pages={614--628},
	issn={0008-4395},
}


\bib{ES1981}{article}{
	author={Efron, B.},
	author={Stein, C.},
	title={The jackknife estimate of variance},
	journal={Ann. Statist.},
	volume={9},
	date={1981},
	number={3},
	pages={586--596},
	issn={0090-5364},
}

\bib{FTGF2015}{article}{
	author={Fejes T\'oth, G.},
	author={Fodor, F.},
	title={Dowker-type theorems for hyperconvex discs},
	journal={Period. Math. Hungar.},
	volume={70},
	date={2015},
	number={2},
	pages={131--144},
	issn={0031-5303},
}

\bib{FHZ16}{article}{
	author={Fodor, Ferenc},
	author={Hug, Daniel},
	author={Ziebarth, Ines},
	title={The volume of random polytopes circumscribed around a convex body},
	journal={Mathematika},
	volume={62},
	date={2016},
	number={1},
	pages={283--306},
	issn={0025-5793},
}

\bib{FoKeVi2014}{article}{
	author={Fodor, F.},
	author={Kevei, P.},
	author={V\'\i gh, V.},
	title={On random disc polygons in smooth convex discs},
	journal={Adv. in Appl. Probab.},
	volume={46},
	date={2014},
	number={4},
	pages={899--918},
	issn={0001-8678},
}


\bib{FoKuVi16}{article}{
author={F. Fodor},
author={\'A. Kurusa},
author={V. V\'igh},
title={Inequalities for hyperconvex sets},
journal={Adv. in Geom.},
volume={16},
date={2016},
number={3},
pages={337--348},

}

\bib{FoVi2012}{article}{
author={F. Fodor},
author={V. V\'igh},
title={Disc-polygonal approximations of planar spindle convex sets},
journal={Acta Sci. Math. (Szeged)},
volume={78},
date={2012},
number={1-2},
pages={331--350},
}

\bib{Gr97}{article}{
   author={Gruber, Peter M.},
   title={Comparisons of best and random approximation of convex bodies by
   polytopes},
   note={II International Conference in ``Stochastic Geometry, Convex Bodies
   and Empirical Measures'' (Agrigento, 1996)},
   journal={Rend. Circ. Mat. Palermo (2) Suppl.},
   number={50},
   date={1997},
   pages={189--216},
}




\bib{R2003}{article}{
   author={Reitzner, Matthias},
   title={Random polytopes and the Efron-Stein jackknife inequality},
   journal={Ann. Probab.},
   volume={31},
   date={2003},
   number={4},
   pages={2136--2166},
   issn={0091-1798},
}

\bib{Rei05}{article}{
	author={Reitzner, Matthias},
	title={Central limit theorems for random polytopes},
	journal={Probab. Theory Related Fields},
	volume={133},
	date={2005},
	number={4},
	pages={483--507},
	issn={0178-8051},
}

\bib{ReSu63}{article}{
   author={A. R\'enyi},
   author={R. Sulanke},
title={\"Uber die konvexe H\"ulle von n zuf\"allig gew\"ahlten Punkten},
journal={Z. Wahrscheinlichkeitsth. verw. Geb.},
volume={2},
date={1963},
pages={75--84.},
}

\bib{ReSu64}{article}{
   author={A. R\'enyi},
   author={R. Sulanke},
 title={\"Uber die konvexe H\"ulle von n zuf\"allig gew\"ahlten Punkten, II.},
 journal={Z. Wahrscheinlichkeitsth. verw. Geb.},
volume={3},
date={1964},
pages={138--147.},
}


\bib{sch08}{article}{
	author={R. Schneider},
	title={Recent results on random polytopes. (Survey)},
	journal={ Boll. Un. Mat. Ital., Ser. (9)},
	volume={1},
	date={2008)},
	pages={17 -- 39.},
}


\bib{Sch14}{book}{
	author={Schneider, Rolf},
	title={Convex bodies: the Brunn-Minkowski theory},
	series={Encyclopedia of Mathematics and its Applications},
	volume={151},
	edition={Second expanded edition},
	publisher={Cambridge University Press, Cambridge},
	date={2014},
	pages={xxii+736},
	isbn={978-1-107-60101-7},
}

\bib{Sch17}{article}{
	author={Schneider, Rolf},
	title={Discrete aspects of stochastic geometry},
	conference={
		title={Handbook of Discrete and Computational Geometry, 3rd ed.},
	},
	book={
		publisher={CRC Press},
		place={Boca Raton},
		date={2017},
		pages={299--329},
	},	
}

\bib{schw09}{book}{
	author={Schneider, Rolf},
	author={Weil, Wolfgang},
	title={Stochastic and integral geometry},
	series={Probability and its Applications (New York)},
	publisher={Springer-Verlag, Berlin},
	date={2008},
	pages={xii+693},
	isbn={978-3-540-78858-4},
}


\bib{SY}{article}{
	author={Schreiber, T.},
	author={Yukich, J. E.},
	title={Variance asymptotics and central limit theorems for generalized
		growth processes with applications to convex hulls and maximal points},
	journal={Ann. Probab.},
	volume={36},
	date={2008},
	number={1},
	pages={363--396},
	issn={0091-1798},
}

\bib{Vu05}{article}{
	author={Vu, V. H.},
	title={Sharp concentration of random polytopes},
	journal={Geom. Funct. Anal.},
	volume={15},
	date={2005},
	number={6},
	pages={1284--1318},
	issn={1016-443X},
}

\bib{TTW17}{article}{
	author={Thaele, Christoph},
	author={Turchi, Nicola},
	author={Wespi, Florian},
	title={Random polytopes: central limit theorems for intrinsic volumes},
	journal={Proc. Amer. Math. Soc.},
	status={to appear},
	date={2017},
	}

\bib{TW18+}{article}{
	author={Turchi, Nicola},
	author={Wespi, Florian},
	title={Limit theorems for random polytopes with vertices on convex surfaces},
	journal={arXiv:1706.02944},
	date={2017},
}

\bib{Vu05}{article}{
	author={Vu, V. H.},
	title={Sharp concentration of random polytopes},
	journal={Geom. Funct. Anal.},
	volume={15},
	date={2005},
	number={6},
	pages={1284--1318},
	issn={1016-443X},
}

\bib{Vu06}{article}{
	author={Vu, Van},
	title={Central limit theorems for random polytopes in a smooth convex
		set},
	journal={Adv. Math.},
	volume={207},
	date={2006},
	number={1},
	pages={221--243},
	issn={0001-8708},
}

\bib{WeWi93}{article}{
   author={Weil, Wolfgang},
   author={Wieacker, John A.},
   title={Stochastic geometry},
   conference={
      title={Handbook of convex geometry, Vol.\ A, B},
   },
   book={
      publisher={North-Holland},
      place={Amsterdam},
   },
   date={1993},
   pages={1391--1438},
}

\end{biblist}
\end{bibdiv}
\end{document}